\documentclass[twocolumn]{autart}  

\usepackage[utf8]{inputenc}
\usepackage[english]{babel}
\usepackage{graphicx,overpic,tikz}     
\usepackage{amsmath}
\usepackage{amssymb}
\usepackage{color}
\usepackage{float}
\usepackage{mathrsfs}
\usepackage{csquotes}
\usepackage{hyperref}
\usepackage{subfigure} 

\usepackage{pgfplots}

\makeatletter
\let\theoremstyle\@undefined                        
\makeatother

\usepackage{amsthm}

\usepackage{mathptmx}
\def\epsilon{\varepsilon}

\newtheorem{Theo}{Theorem}

\newtheorem{Lem}{Lemma}

\usepackage{floatflt}

\begin{document}
\begin{frontmatter}

\title{ Simultaneous compensation of input delay and state/input quantization
	for linear systems via switched predictor feedback} 

\thanks[footnoteinfo]{	Funded by the European Union (ERC, C-NORA, 101088147). Views and opinions expressed are however those of the authors only and do not necessarily reflect those of the European Union or the European Research Council Executive Agency. Neither the European Union nor the granting authority can be held responsible for them.}
\author[TUC]{Florent Koudohode}   \ead{ fkoudohode@tuc.gr}  and   
\author[TUC]{Nikolaos Bekiaris-Liberis} \ead{bekiaris-liberis@ece.tuc.gr}             

\address[TUC]{Department of Electrical and Computer Engineering, Technical University of Crete, 73100 Chania, Greece}


\begin{abstract}                          
We develop a switched predictor-feedback law, which achieves global asymptotic stabilization of linear systems with input delay and with the plant and actuator states available only in (almost) quantized form. The control design relies on a quantized version of the nominal predictor-feedback law for linear systems, in which quantized measurements of the plant and actuator states enter the predictor state formula. A switching strategy is constructed to dynamically adjust the tunable parameter of the quantizer (in a piecewise constant manner), in order to initially increase the range and subsequently decrease the error of the quantizers. The key element in the proof of global asymptotic stability in the supremum norm of the actuator state is derivation of solutions'  estimates combining a backstepping transformation with small-gain and input-to-state stability arguments, for addressing the error due to quantization. We extend this result to the input quantization case and illustrate our theory with a numerical example.
\end{abstract}

\end{frontmatter}
\section{Introduction}
Compensation of long input delays for linear systems can be achieved via predictor-based control design techniques and, in particular, via exact predictor feedbacks, see, for example, \cite{artstein1982linear,bekiaris2013nonlinear,deng2022predictor,krstic2009delay}. The baseline, continuous predictor-based designs are accompanied with certain stability (and robustness) guarantees, see, e.g., \cite{bekiaris2013nonlinear,bresch2014prediction,karafyllis2017predictor,krstic2009delay}. However, implementation of predictor feedbacks may be subject to digital effects, such as, for example, sampling and quantization, which may deteriorate closed loop performance of the nominal continuous designs, when these effects are left uncompensated, see, for example, \cite{karafyllis2017predictor,mazenc2016predictor}. Therefore, addressing issues arising due to digital implementation of predictor feedbacks is practically and theoretically significant, in order to provably preserve the stability guarantees of the original, continuous designs. 

Among the potential digital implementation issues arising in implementation of predictor feedbacks, sampled measurements and control inputs applied via zero-order hold have been addressed in \cite{karafyllis2011nonlinear,selivanov2016observer,selivanov2016predictor}; while \cite{battilotti2019continuous,weston2018sequential} address sampling in sequential predictors-based designs and \cite{di2019sampled,di2022semi} address sampling in nonlinear systems with state delay. In particular, \cite{karafyllis2011nonlinear,selivanov2016observer,selivanov2016predictor} introduce design and analysis approaches for compensating the effect of sampling in measurements and actuation employing predictor-based designs; while \cite{di2020practical} also addresses quantization effects in nonlinear systems with delay on the state. Because our design also relies on a switched strategy (although it is assumed that the input is continuously applied and that continuous measurements are available), results on predictor-based event-triggered control design may be also viewed as relevant \cite{gonzalez2019event,mazenc2016predictor,mazenc2022event,nozari2020event,sun2022predictor}. Besides \cite{bekiaris2020hybrid} that considers the case of a class of boundary controlled, first-order hyperbolic PDEs (Partial Differential Equations) subject to state quantization, not involving an ODE (Ordinary Differential Equation) part, paper \cite{schlanbusch2022attitude} presents a predictor-feedback design for a particular case of a model of robot manipulator with quantized input, which, however, neither addresses a general linear system with a systematic design and analysis approach nor aims at achieving asymptotic stability via a dynamic quantizer, while quantization affects the control input and not the state measurements.  

In the present paper, we develop a switched predictor-feedback control design that achieves compensation of both input delay and state (or input) quantization. The control design relies on two main ingredients—a baseline predictor-feedback design and a switching strategy that dynamically adjusts the tunable parameter of the quantizer. In particular, the nominal predictor state formula is modified such that the plant and actuator states are replaced by their quantized versions; while the tunable parameter of the quantizer is dynamically adjusted (in a piecewise manner), in order to initially increase the range of the quantizer and to subsequently decreases its error, as it is done in the case of delay-free systems in \cite{brockett2000quantized,liberzon2003hybrid}. The transition between the two main modes of operation of the switching signal takes place at detection of an event indicating that the infinite-dimensional state of the system (consisting of the ODE and actuator states) enters the range of the quantizer, implying that stabilization can be then achieved decreasing the quantization error. 

We establish global asymptotic stability in the supremum norm of the actuator state. The proof strategy relies on combination of backstepping \cite{krstic2009delay} with small-gain and input-to-state stability (ISS) arguments \cite{karafyllis2021input}, towards derivation of estimates on solutions. In particular, the proof consists of two main steps. In the first, the system operates in open loop, while the range of the quantizer is increasing. Deriving estimates based on the explicit solution of the open-loop system, it is shown that there exists some time instant at which the state enters within the range of the quantizer. In the second step of the proof, given that the state is within the quantizer's range, the tunable parameter of the quantizer is decreasing in a piecewise constant manner. In particular, within each interval in which the quantizer's parameter is constant, we show that the norm of the solutions of the closed-loop system decreases by a factor that is less than unity. To show this we capitalize on the input-to-state stability properties of linear predictor feedbacks and of the pure-transport PDE system, which allow us to obtain, via a small-gain argument, a condition on the quantizer's parameters, namely, on its range and error, which guarantees that asymptotic stability is achieved. We note that the quantizers considered here are called \enquote{almost} quantizers. The reason is that we use functions that are locally Lipschitz and not just piecewise constant functions to avoid issues related to existence and uniqueness of solutions\footnote{Arising due to the potential non-measurability of composition of an exact quantizer function with an only continuous PDE state.}. This assumption allows us to study existence and uniqueness using the results from \cite{karafyllis2021input} in combination with \cite{espitia2017stabilization,KKnonlocal}, while the stability estimates derived and the overall proof strategy adopted do not depend on this property, suggesting that such an assumption can be removed. 

We start in Section~\ref{probFormulation} presenting the classes of systems and quantizers considered, together with the switched predictor-feedback design. In Section~\ref{stabStateQuantization} we establish global asymptotic stability of the closed-loop system in the state quantization case and this result is extended to the input quantization case in Section~\ref{inputquantization}. In Section~\ref{simulation} we present consistent simulation results. In Section~\ref{conlusion} we provide concluding remarks and a discussion on potential future research extensions. 

{\em Notation:} We denote by $L^{\infty}(A ; \Omega)$ the space of measurable and bounded functions defined on $A$ and taking values in $\Omega$. For a given $D>0$ and a function  $u \in L^{\infty}([0, D] ; \Rset)$ we define $\|u\|_{\infty}=\sup _{x \in[0, D]}|u(x)|$.  For a given $h \in \Rset$ we define its integer part as $\lfloor h\rfloor=\max \{k \in \Zset: k < h\}$. The state space $\Rset^n\times L^{\infty}([0, D] ; \Rset) $ is induced with norm $\|(X,u)\|=|X|+\| u\|_{\infty}.$ We denote by $AC\left(\mathbb{R}_{+}, \mathbb{R}^{n}\right),$ the set of all absolutely continuous function $X:\mathbb{R}_+\to \mathbb{R}^n.$ Let $I \subseteq \mathbb{R}$ be an interval. A piecewise left-continuous function (resp. a piecewise right-continuous function) $f : I \rightarrow J$ is a function continuous on each closed interval subset of $I$ except possibly on a finite number of points $x_0 < x_1 < \cdot \cdot \cdot < x_p$ such that for all $l \in \{0, \cdot \cdot \cdot, p - 1\}$ there exists $f_l$ continuous on $[x_l, x_{l+1}]$ and $f_l |_{(x_l, x_{l+1})} = f|_{(x_l, x_{l+1})}$. Moreover, at the points $x_0, \cdot \cdot \cdot, x_p$ the function is left continuous. The set of all piecewise left-continuous functions is denoted by $\mathcal{C}_{lpw} (I, J)$ (see also \cite{espitia2017stabilization,KKnonlocal}).

\section{ Problem Formulation and Control Design}\label{probFormulation}
\subsection{ Linear Systems With Input Delay \& State Quantization}
We consider the following system 
\begin{equation} \label{linear_delay_system}
	\dot{X}(t)=A X(t)+B U(t-D),
\end{equation}
where $D>0$ is constant input delay, $t \geq 0$ is time variable, $X \in \Rset^{n}$ is state, and $U$ is scalar control input. An alternative representation of this system is as follows 
\begin{align}
	\label{pde_representation}	\dot{X}(t)&=A X(t)+B u(0,t),\\
	\label{pde_representation01}	u_{t}(x, t)&=u_{x}(x,t), \\
	\label{pde_representation1}	u(D, t)&=U(t),
\end{align}
by setting $u(x, t)=U(t+x-D),$ where $ x\in [0, D]$ and $u$  is the transport PDE state, with initial conditions $u(x,0) = u_0 (x)$. 
We proceed from now on with representation \eqref{pde_representation}--\eqref{pde_representation1} as it turns out to be more convenient for control design and analysis. In \cite{krstic2009delay} system \eqref{pde_representation}--\eqref{pde_representation1} is transformed into 
\begin{align}\label{targetsystem_without_quantizer}
	\dot{X}(t)&=(A+B K) X(t)+Bw(0,t), \\
	\label{targetsystem_without_quantizer2}	w_{t}(x, t)&=w_{x}(x, t), \\
	w(D, t)&=U(t)-U_{\rm nom}(t),\label{targetsystem_without_quantizer3}
\end{align} thanks to the backstepping transformation
\begin{align}\label{backstepping_direct_transformation}
	w(x, t)&=u(x,t)-K\int_{0}^{x} e^{A(x-y)}Bu(y,t) d y-Ke^{A x} X(t),
\end{align}
where $U_{\rm nom}(t)$ is the nominal predictor feedback defined as follows \begin{align}\label{nominalU}
	U_{\rm nom}(t)&=K\int_{0}^{D} e^{A(D-y)}Bu(y,t)dy+Ke^{A D} X(t).
\end{align}
The inverse of this transformation is 
\begin{align}\label{backstepping_inverse_transformation}
	u(x,t)&=w(x,t)+K\int_{0}^{x} e^{(A+BK)(x-y)} B w(y, t) d y\\
	\nonumber&+Ke^{(A+BK) x} X(t).
\end{align} 
One has  \begin{equation}\label{equivalence}
	M_2\|(X,u)\|\le \|(X,w) \|\le M_1\|(X,u)\|,
\end{equation}
where $M_1,M_2$ are 
\begin{align}
	\label{M1}	M_{1} & =|K| \mathrm{e}^{|A| D} \max \{1, D|B|\}+1, \\
	\label{M2}	M_{2} & =\frac{1}{|K| \mathrm{e}^{|A+B K| D} \max \{1, D|B|\}+1}. 
\end{align}
Although \eqref{backstepping_direct_transformation}--\eqref{M2} are well-known facts, we present them here as the constants $M_1$ and $M_2$ are incorporated in the control design.
\subsection{Properties of the Quantizer}
The state $X$ of the plant and the actuator state $u$ are available only in quantized form. We consider here dynamic quantizers with an adjustable parameter of the form (see, e.g., \cite{brockett2000quantized,liberzon2003hybrid})
\begin{equation}\label{quantizer}
	q_{\mu}(X,u)=( q_{1\mu}(X),q_{2\mu}(u))=\left(\mu q_1\left(\frac{X}{\mu}\right),\mu q_2\left(\frac{u}{\mu}\right)\right),
\end{equation} where $\mu>0$ can be manipulated and this is called \enquote{zoom} variable. The quantizers $q_1:\mathbb{R}^n\to \mathbb{R}^n$ and $q_2:L^{\infty}([0,D];\mathbb{R})\to L^{\infty}([0,D];\mathbb{R})$ are locally Lipchitz functions that satisfy the following properties

P1: If $\|(X,u)\| \leq M$, then $\|(q_1(X)-X, q_2(u)-u)\| \leq \Delta$,\\
P2: If $\|(X,u)\|>M$, then $\|(q_1(X),q_2(u))\|>M-\Delta$,\\
P3: If $\|(X,u)\| \leq \hat{M}$, then $q_1(X)=0$ and $q_2(u)=0,$ 

for some positive constants $M, \hat{M}$, and $\Delta$, with $M>\Delta$ and $\hat{M}<M$. When the argument of the employed quantizer is a vector, the quantizer function is a vector itself, defined componentwise according to \eqref{quantizer}, satisfying properties $\rm P1$--$\rm P3$. In the present case, we consider uniform quantizers for each element of the vector argument, while we discern between the quantizer function of the measurements of the plant state $X$ and the function that corresponds to the actuator state measurements $u$. For simplicity of control design and analysis we assume a single tunable parameter $\mu$. In fact, this is also practically reasonable, considering a case where, e.g., a single computer with a single camera collects measurements.
\begin{figure}
	\begin{center}
		\includegraphics[width=8cm]{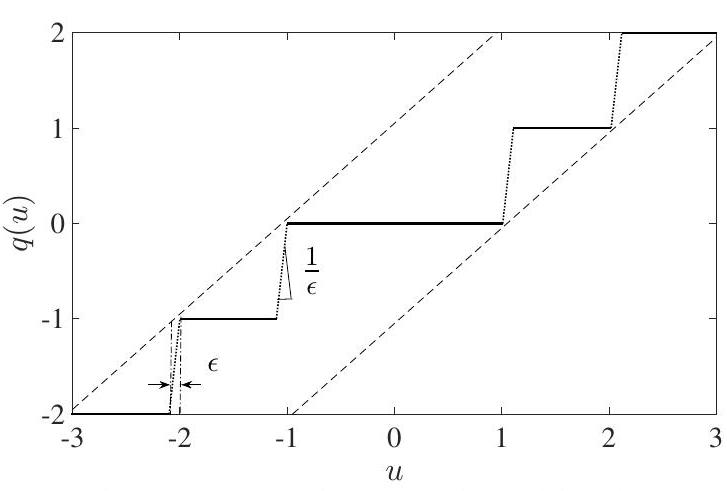}
		\caption{An approximate quantizer with $\epsilon$-layer.} 
		\label{an_approximate_quantizer}                                  
	\end{center}                                 
\end{figure}
The quantizers considered here differ from typical piecewise constant quantizers, taking finitely many values, in that we assume they are locally Lipschitz functions. Although this may appear as a restrictive requirement, in practice, it isn't as it is also illustrated in Figure~\ref{an_approximate_quantizer}, which shows a locally Lipschitz quantizer that may arbitrarily closely approximate a quantizer with rectilinear quantization regions. This is a technical requirement to guarantee existence and uniqueness of solutions in a straightforward manner. The stability result obtained and the stability proof do not essentially rely on this, suggesting that this assumption could be removed.

\subsection{Predictor-Feedback Law Using Almost Quantized Measurements}
The hybrid predictor-feedback law can be viewed as a quantized version of the predictor-feedback controller \eqref{nominalU}, in which the dynamic quantizer depends on a suitably chosen piecewise constant signal $\mu$. It is defined as \begin{equation}\label{control_quantizer}
	U(t)=\left\{\begin{array}{ll}0, & 0 \leq t \leq t_{0} \\ K P_{\mu(t)}\left(X(t), u(\cdot,t)\right), & t>t_{0}
	\end{array}\right.,
\end{equation} where\begin{equation}\label{predictor_quantizer}
	P_{\mu}(X, u)=\mathrm{e}^{A D} q_{1\mu}(X)+\int_{0}^{D} \mathrm{e}^{A (D-y)} B q_{2\mu}(u(y)) dy,
\end{equation} and $K$ is a gain vector that makes matrix $A+B K$ Hurwitz. The tunable parameter $\mu$ is selected as 
\begin{equation}\label{switching_parameter}
	\mu(t)= \begin{cases}\overline{M}_1 \mathrm{e}^{2|A| j \tau} \mu_{0}, & (j-1) \tau \leq t \leq j \tau, \\ & 1 \leq j \leq\left\lfloor\frac{t_{0}}{\tau}\right\rfloor+1,\\ \mu\left(t_{0}\right), & t \in\left(t_{0}, t_{0}+T\right], \\ \Omega \mu\left(t_{0}+(i-1) T\right), & t \in\left(t_{0}+(i-1) T,\right. \\ & \left.t_{0}+i T\right], \quad i=2,3, \ldots\end{cases},
\end{equation} for some fixed, yet arbitrary, $\tau, \mu_0>0$, with $t_{0}$ being the first time instant at which the following holds 
\begin{align}
	\nonumber&\left|\mu(t_0)q_1\left(\frac{X\left(t_{0}\right)}{\mu\left(t_{0}\right)}\right)\right|+\left\|\mu(t_0)q_2\left(\frac{u\left( t_{0}\right)}{\mu\left(t_{0}\right)}\right)\right\|_{\infty}\\
	&	\label{first_time_t0}\leq (M \overline{M}- \Delta) \mu(t_0),
\end{align} where
\begin{align}
	\label{M3}	M_{3} & =|K| \mathrm{e}^{|A| D}(1+|B| D), \\
	\label{M}	\overline{M} & =\frac{M_{2}}{M_1(1+M_0)}, \\
	\label{Mbar} \overline{M}_1&=1+D|B|,\\
	\label{Omega}	\Omega & =\frac{(1+\lambda) (1+M_0)^2 \Delta M_{3}}{M_{2} M}, \\
	\label{T}	T & =-\frac{\ln \left(\frac{\Omega}{1+M_0} \right)}{\delta}.
\end{align} The parameters $\delta, \lambda$ and $M_0$ are defined as follows. Constant $\delta \in(0, \min\{\sigma,\nu\}),$ for some $\nu,\sigma>0,$ satisfying
\begin{equation} \label{expABK}
	\left|e^{(A+B K) t}\right| \leq M_{\sigma} e^{-\sigma t}, 
\end{equation} for some $M_{\sigma}>1$, $\lambda$ is selected large enough in such a way that the following small-gain condition holds
\begin{equation}\label{small_gain}
	\frac{e^{D}}{1+\lambda}\left(\frac{M_{\sigma}}{\sigma}|B|+1\right)<1,	
\end{equation} and $M_0$ is defined such that
\begin{align}
	\nonumber& M_{0}=\max \left\{\left(1-\phi\right)^{-1}\left(1-\varphi_{1}\right)^{-1}e^D ;\right. \\
	\nonumber	&\left.\left(1-\phi\right)^{-1} \left(1-\varphi_{1}\right)^{-1}\phi M_{\sigma}\right\}+ \max \left\{(1-\varphi_1)^{-1}M_{\sigma}; \right. \\ 	\label{M0}&\left.(1+\varepsilon) \left(1-\phi\right)^{-1}(1-\varphi_1)^{-1}e^D\frac{M_{\sigma}}{\sigma}|B|\right\},
\end{align} where $0<\phi<1$ and $0<\varphi_{1}<1$ with
\begin{align} \label{phi and psy}
	\phi=\frac{1+\varepsilon}{1+\lambda} e^{D(\nu+1)} \text{ and }\varphi_{1}=(1+\varepsilon)(1-\phi)^{-1} \phi \frac{M_{\sigma}}{\sigma}|B|,
\end{align} for some  $\varepsilon>0$.  The choice of $\nu,\varepsilon$ guarantees that $\phi<1,\varphi_1<1$, which is always possible given \eqref{small_gain} (see also the proof of Lemma~\ref{Lemma2} in Section~\ref{stabStateQuantization}). 

We note here few remarks for the control law presented. Event \eqref{first_time_t0} can be detected using measurements of $q_{1\mu}(X), q_{2\mu}(u)$, and $\mu$ only, which are available. The tunable parameter $\mu$ is chosen in a piecewise constant manner. In the first phase (for $0\le t\le t_0$) it is increasing sufficiently fast, choosing a sufficiently large $\overline{M}_1$, so that there exists a time $t_0$ such that \eqref{first_time_t0} holds (see Lemma~
\ref{Lemma1}) which depends on the size of open-loop solutions. In the zooming in phase, $\mu$ is decreasing by a factor of $\Omega$ over $T$ time units. To guarantee convergence of the state to zero, $\Omega$ has to be less than one, which is guaranteed by assumption (see Theorem~\ref{Theorem1}). Time $T$ has to be chosen large enough, as $T$ time units intervals correspond to intervals in which the solutions approach the equilibrium. Condtion \eqref{small_gain} is a small-gain condition derived when viewing as disturbance the error due to quantized measurements and applying an ISS argument (see Lemma~
\ref{Lemma2}). In particular, when $D$ increases, $\lambda$ has to increase, which (via $\Omega$ in \eqref{Omega}) imposes stricter conditions on the ratio $\frac{\Delta}{M}$ between error and range of the quantizer.
\section{Stability of Switched Predictor-Feedback Controller Under State Quantization}\label{stabStateQuantization}
\begin{Theo}\label{Theorem1}
	Consider the closed-loop system consisting of the plant \eqref{pde_representation}--\eqref{pde_representation1} and the switched predictor-feedback law \eqref{control_quantizer}--\eqref{switching_parameter}. Let the pair $(A, B)$ be stabilizable. If $\Delta$ and $M$ satisfy  \begin{equation}\label{conditionMDelta}
		\frac{\Delta}{M}<\frac{M_2}{(1+M_0)\max \{M_3(1+\lambda)(1+M_0),2M_1\}},
	\end{equation}	then for all $X_{0} \in \mathbb{R}^{n}$, $u_{0} \in \mathcal{C}_{lpw}([0, D], \mathbb{R})$, there exists a unique solution such that $X(t) \in AC\left(\mathbb{R}_{+}, \mathbb{R}^{n}\right)$, for each $t \in \mathbb{R}_{+}$ $u(\cdot, t) \in \mathcal{C}_{lpw}([0,D], \mathbb{R})$, and for each $x \in[0,D]$ $u(x,\cdot) \in \mathcal{C}_{\text {lpw }}\left(\mathbb{R}_{+}, \mathbb{R}\right)$, which satisfies
	\begin{align}
		&|X(t)|+\left\|u(t)\right\|_{\infty}
		\label{stability_result_u} \leq  \gamma\left(|X_0|+\left\|u_{0}\right\|_{\infty}\right)^{\left(2-\frac{\ln \Omega}{T} \frac{1}{|A|}\right)} \mathrm{e}^{\frac{\ln \Omega}{T}t},
	\end{align}	where
	\begin{align}
		\nonumber	\gamma&=\frac{\overline{M}_1}{M_2}\max \left\{\frac{M_{2}M}{\Omega} e^{2|A| \tau} \mu_{0}, M_1\right\} \\
		\nonumber&
		\times \max \left\{\frac{1}{\mu_0(M \overline{M}-2 \Delta)}, 1\right\}\\
		&\times\left(\frac{1}{\mu_0(M \overline{M}-2 \Delta)}\right)^{\left(1-\frac{\ln \Omega}{T} \frac{1}{|A|}\right)}.
	\end{align} 
	
\end{Theo}
The proof relies on Lemmas \ref{Lemma1} and \ref{Lemma2} which are presented next. In particular, Lemma~\ref{Lemma1} establishes a bound on the solutions during the zooming out (open-loop) phase, which in turn is utilized to prove the existence of a time instant at which the solutions get within the range of the quantizer (ie., they satisfy \eqref{first_time_t0}). Subsequently, Lemma~\ref{Lemma2} establishes a bound on closed-loop solutions, via employing a small-gain based ISS argument, for time intervals of constant $\mu$. This bound implies that the solutions' magnitude decays by a factor of $\Omega$ every $T$ time units.
\vspace{.5cm}
\begin{Lem}\label{Lemma1}
	Let $\Delta$ and $M$ satisfy \eqref{conditionMDelta}, 
	there exists a time $t_0$ satisfying 
	\begin{equation}\label{t0}
		t_{0} \leqslant \frac{1}{|A|} \ln\left(\dfrac{\frac{1}{\mu_{0}}\left(\left|X_{0}\right|+\left\|u_{0}\right\|_{\infty}\right)}{(M \overline{M}-2 \Delta)}\right),
	\end{equation} 
	such that \eqref{first_time_t0} holds, and thus, the following also holds 
	\begin{equation}\label{bound_in_time_t0}
		|X(t_{0})|+\|u(t_0)\|_{\infty} \leq  M \overline{M}\mu(t_0).
	\end{equation}
\end{Lem}
\begin{proof}	
	For all $0\le t\le t_0$, thanks to \eqref{control_quantizer} one has $U(t)=0,$ and thus, the corresponding open-loop system reads as
	\begin{align}\label{pde_representation0}
		\dot{X}(t)&=A X(t)+B u(0,t),\\
		u_{t}(x, t)&=u_{x}(x,t), \\
		\label{pde_representation3}	u(D, t)&=0.
	\end{align}	Using the method of characteristics, the solution to the transport subsystem reads as $u(x, t)=u_{0}(x+t)$ for $0\le x+t\le D$ and $u(x,t)=0$ for $x+t>D$. Thus 
	\begin{equation}\label{estimation_u_norm}
		\|u(t)\|_{\infty} \leqslant \left\|u_{0}\right\|_{\infty}. 
	\end{equation}
	From the equation \eqref{estimation_u_norm}
	one has by the variation of constants formula for $0\le t\le t_0$
	\begin{align}
		|X(t)|&\leq\overline{M}_1 e^{|A| t}\left(\left|X_{0}\right|+\left\|u_{0}\right\|_{\infty}\right),\label{normXt0}
	\end{align} with $\overline{M}_1$ given by \eqref{Mbar}.
	Therefore, combining \eqref{estimation_u_norm} and \eqref{normXt0}  we have for $0\le t\le t_0$
	\begin{equation}\label{normX0t0}
		|X(t)|+\|u(t)\|_{\infty} \leqslant \overline{M}_1e^{|A| t}\left(\left|X_{0}\right|+\left\|u_{0}\right\|_{\infty}\right).
	\end{equation} Choosing the switching signal $\mu$ according to \eqref{switching_parameter}, one has the existence of a time $t_0$ verifying \eqref{t0}  such that
	\begin{align} 
		\dfrac{\left|X\left(t_{0}\right)\right|}{\mu(t_0)}+ \dfrac{\left\|u\left(t_{0}\right)\right\|_{\infty}}{\mu(t_0)} \le M \overline{M}-2 \Delta.
	\end{align} Thus, using property P1 of the quantizer and the triangle inequality one obtains
	\begin{align}
		\nonumber&\left|q_1\left(\frac{X(t_0)}{\mu(t_0)}\right)\right|+\left\|q_2\left(\frac{u\left( t_{0}\right)}{\mu\left(t_{0}\right)}\right)\right\|_{\infty}\\
		\nonumber&=\left|q_1\left(\frac{X(t_0)}{\mu(t_0)}\right)-\frac{X(t_0)}{\mu(t_0)}+\frac{X(t_0)}{\mu(t_0)}\right|\\
		\nonumber	&+\left\|q_2\left(\frac{u\left( t_{0}\right)}{\mu\left(t_{0}\right)}\right)-\frac{u\left( t_{0}\right)}{\mu\left(t_{0}\right)}+\frac{u\left( t_{0}\right)}{\mu\left(t_{0}\right)}\right\|_{\infty}\\
		\nonumber	&\le \left|q_1\left(\frac{X(t_0)}{\mu(t_0)}\right)-\frac{X(t_0)}{\mu(t_0)}\right|+\left\|q_2\left(\frac{u\left( t_{0}\right)}{\mu\left(t_{0}\right)}\right)-\frac{u\left( t_{0}\right)}{\mu\left(t_{0}\right)}\right\|_{\infty}\\
		\nonumber&+\dfrac{\left|X\left(t_{0}\right)\right|}{\mu(t_0)}+ \dfrac{\left\|u\left(t_{0}\right)\right\|_{\infty}}{\mu(t_0)}\\
		\label{relation_with_q}
		&\le M \overline{M}- \Delta. 
	\end{align} This implies that the relation \eqref{first_time_t0} is satisfied. We will then prove that detecting event \eqref{first_time_t0} along with the properties of the quantizer confirms that at time $t_0$, relation \eqref{bound_in_time_t0} also holds. To do so, let us first exclude that $ \left|\frac{X\left(t_{0}\right)}{\mu\left(t_{0}\right)}\right|+\left\|\frac{ u\left(t_{0}\right)}{\mu\left(t_{0}\right)}\right\|_{\infty}> M$ is satisfied. If that was the case, thanks to property P2 of the quantizer the inequality
	\begin{align}
		&\left|q_1\left(\frac{X\left(t_{0}\right)}{\mu\left(t_{0}\right)}\right)\right|+ \left\|q_2\left(\frac{u\left( t_{0}\right)}{\mu\left(t_{0}\right)}\right)\right\|_{\infty}>M-\Delta,
	\end{align} would be verified, and therefore, since $\overline{M}\le 1$
	\begin{align}
		&\left|q_1\left(\frac{X\left(t_{0}\right)}{\mu\left(t_{0}\right)}\right)\right|+ \left\|q_2\left(\frac{u\left( t_{0}\right)}{\mu\left(t_{0}\right)}\right)\right\|_{\infty}>M\overline{M}-\Delta,
	\end{align} would also be satisfied. This contradicts \eqref{first_time_t0}. Thus, $\left|\frac{X\left(t_{0}\right)}{\mu\left(t_{0}\right)}\right|+\left\|\frac{ u\left(t_{0}\right)}{\mu\left(t_{0}\right)}\right\|_{\infty}\le M$ holds.  Let us check now if the following could hold
	\begin{equation}\label{eqMMbar}
		M\overline{M}< \left|\frac{X\left(t_{0}\right)}{\mu\left(t_{0}\right)}\right|+\left\|\frac{ u\left(t_{0}\right)}{\mu\left(t_{0}\right)}\right\|_{\infty} \le M. 
	\end{equation}  In that case, with property P1 of the quantizer one would have
	\begin{align}
		\nonumber&\left|q_1\left(\frac{X\left(t_{0}\right)}{\mu\left(t_{0}\right)}\right)-\frac{X\left(t_{0}\right)}{\mu\left(t_{0}\right)}\right|+\left\|q_2\left(\frac{u\left(t_{0}\right)}{\mu\left(t_{0}\right)}\right)-\frac{u\left( t_{0}\right)}{\mu\left(t_{0}\right)}\right\|_{\infty}\\
		& \leqslant \Delta. \label{eqMbar1}
	\end{align} Using the triangle inequality one has from \eqref{eqMbar1}
	\begin{align}
		\nonumber&\left|\frac{X\left(t_{0}\right)}{\mu\left(t_{0}\right)}\right|+\left\|\frac{u\left(t_{0}\right)}{\mu\left(t_{0}\right)}\right\|_{\infty}\\
		\nonumber&\le \left|q_1\left(\frac{X\left(t_{0}\right)}{\mu\left(t_{0}\right)}\right)-\frac{X\left(t_{0}\right)}{\mu\left(t_{0}\right)}\right|+\left\|q_2\left(\frac{u\left( t_{0}\right)}{\mu\left(t_{0}\right)}\right)-\frac{u\left( t_{0}\right)}{\mu\left(t_{0}\right)}\right\|_{\infty}\\
		\nonumber&+\left|q_1\left(\frac{X\left(t_{0}\right)}{\mu\left(t_{0}\right)}\right)\right|+ \left\|q_2\left(\frac{u\left( t_{0}\right)}{\mu\left(t_{0}\right)}\right)\right\|_{\infty}\\
		\nonumber &\le \Delta+\left|q_1\left(\frac{X\left(t_{0}\right)}{\mu\left(t_{0}\right)}\right)\right|+ \left\|q_2\left(\frac{u\left( t_{0}\right)}{\mu\left(t_{0}\right)}\right)\right\|_{\infty}.
	\end{align} Then we obtain by using this last inequality and \eqref{eqMMbar}
	\begin{equation}
		\overline{M} M-\Delta<\left|q_1\left(\frac{X\left(t_{0}\right)}{\mu\left(t_{0}\right)}\right)\right|+\left\|q_2\left(\frac{u\left( t_{0}\right)}{\mu\left(t_{0}\right)}\right)\right\|_{\infty},	
	\end{equation} 
	contradicting \eqref{first_time_t0}. Thus,  \eqref{bound_in_time_t0} is also valid.
\end{proof}

\begin{Lem}\label{Lemma2}
	Choose $K$ such that $A+BK$ is Hurwitz and let $\sigma,M_{\sigma}>0$ be such that \eqref{expABK} holds. Select $\lambda$ large enough in such a way that the small-gain condition \eqref{small_gain} holds. Then the solutions to the target system \eqref{targetsystem_without_quantizer}--\eqref{targetsystem_without_quantizer3} with the quantized controller \eqref{control_quantizer}, which verify, for fixed $\mu$,
	\begin{equation}\label{hyplemma2}
		|X(t_0)|+\|w(t_0)\|_{\infty}\le \frac{M_{2}}{1+M_0} M\mu,
	\end{equation} they satisfy for $t_0<t\le t_0+T$
	\begin{align}
		\nonumber& |X(t)|+\|w(t)\|_{\infty}\leqslant \max\left\{ M_{0} e^{-\delta (t-t_0)}\left(\left|X(t_0)\right| \right.\right.\\
		&\left. +\left\|w(t_0)\right\|_{\infty}\right),\left. \Omega \frac{M_{2}}{1+M_0} M\mu \right\}.\label{norm_X_w}
	\end{align} In particular, the following holds
	\begin{align}\label{normXu1}
		|X(t_0+T)|+\|w(t_0+T)\|_{\infty} \leq \Omega \frac{M_{2}}{1+M_0} M\mu.
	\end{align}
\end{Lem}
\begin{proof}
	By virtue of \eqref{small_gain}, since the function
	\begin{equation}
		h\left(\bar{\varepsilon}, \bar{\nu}\right)=\frac{1+\bar{\varepsilon}}{1+\lambda} e^{D\left(\bar{\nu}+1\right)}\left(\frac{M_{\sigma}}{\sigma}|B|(\bar{\varepsilon}+1)+1\right), 
	\end{equation}
	is continuous at $(0,0)$ and verifies $h(0,0)<1$, there exist constants $\varepsilon$ and $\nu>0$ such that $h\left(\varepsilon, \nu\right)<1$, that is
	\begin{equation}\label{small_gain2}
		\frac{1+\varepsilon}{1+\lambda} e^{D(\nu+1)}\left(\frac{M_{\sigma}}{\sigma}|B|(\varepsilon+1)+1\right)<1.
	\end{equation} This condition implies
	\begin{equation}\label{small_gain1}
		\frac{1+\varepsilon}{1+\lambda} e^{D\left(\nu+1\right)}<1.	
	\end{equation}
	Using \eqref{targetsystem_without_quantizer} one obtains with variation of constants  formula
	\begin{equation}\label{solution_X}
		X(t)=e^{(A+B K) t} X(t_0)+\int_{t_0}^{t} e^{(A+B K)(t-s)} B w(0, s) ds.
	\end{equation}	
	Using relation \eqref{expABK}, relation \eqref{solution_X} gives
	\begin{equation}
		|X(t)| \leqslant M_{\sigma} e^{-\sigma (t-t_0)}|X(t_0)|+\frac{M_{\sigma}}{\sigma}|B| \sup_{t_0\leqslant s \leqslant t}\left(\|w(s)\|_{\infty}\right).\label{eqXw1}
	\end{equation}
	Using the fading memory lemma  \cite[Lemma 7.1]{karafyllis2021input},  for any $\varepsilon>0$, there exists $\delta_1\in (0,\sigma)$ such that 	
	\begin{align}
		\nonumber	|X(t)| e^{\delta_1 (t-t_0)}&\leqslant M_{\sigma}|X(t_0)|+(1+\varepsilon) \frac{M_{\sigma}}{\sigma}|B| \\
		&\times\sup_{t_0 \leqslant s \leqslant t}\left(\|w(s)\|_{\infty} e^{\delta_1 (s-t_0)}\right).\label{eqXw2}
	\end{align}	For the transport subsystem \eqref{targetsystem_without_quantizer2}, \eqref{targetsystem_without_quantizer3}, one can invoke the ISS estimate in sup-norm in \cite[estimate (2.23)]{KKnonlocal} (see also \cite[estimate (3.2.11)]{karafyllis2021input}) to get for all $\nu>0$ that
	\begin{align}
			\nonumber	\|w(t)\|_{\infty} &\leq e^{-\nu(t-t_0-D)}e^{D}\left\|w(t_0)\right\|_{\infty}\\
	\label{wmu1d}	&+e^{D\left(1+\nu\right)} \sup _{t_0 \leq s \leq t}(|d(s)|),
	\end{align}where\begin{align}
	\nonumber d(t)&=K e^{A D} \mu(t)\left(q_1\left(\frac{X(t)}{\mu(t)}\right)-\frac{X(t)}{\mu(t)}\right)\\
				\label{d1}	&+K\int_{0}^{D}  e^{A(D-y)} B \mu(t)\left(q_2\left(\frac{u(y, t)}{\mu(t)}\right)-\frac{u(y, t)}{\mu(t)}\right)dy,
	\end{align} with $u$ given in terms of $(X,w)$ by the inverse backstepping transformation \eqref{backstepping_inverse_transformation}. Applying the fading memory inequality \cite[Lemma 7.1]{karafyllis2021input}, there exists $\delta_2\in (0,\nu)$ such that
	\begin{align}
		\nonumber	\|w(t)\|_{\infty} e^{\delta_2 (t-t_0)} &\leqslant e^{D}\left\|w(t_0)\right\|_{\infty}+e^{D\left(\nu+1\right)}(1+\varepsilon) \\&\times\sup _{t_0 \leqslant s \leqslant t}\left(|d(s)| e^{\delta_2 (s-t_0)}\right). \label{eqWd}
	\end{align}
	Let us define $\delta$ as the minimum of $\delta_1$ and $\delta_2,$ 
	thus $\delta \in (0, \min\{\sigma,\nu\})$. Now, for all $t\geq t_0$, let us define the following quantities
	\begin{align}
		\label{defnormw}\|w\|_{[t_0, t]}&:=\sup_{t_0 \leqslant s \leqslant t}\|w(s)\|_{\infty}e^{\delta (s-t_0)} ,\\
		|\label{defnormX}X|_{[t_0, t]}&:=\sup_{t_0 \leqslant s \leqslant t}|X(s)| e^{\delta (s-t_0)}.
	\end{align} We can therefore obtain from \eqref{eqXw2} and \eqref{eqWd},  using the definitions \eqref{defnormw} and \eqref{defnormX}, the following
	\begin{align}\label{norm_x_w}
		|X|_{[t_0, t]} \leqslant M_{\sigma}|X(t_0)|+(1+\varepsilon) \frac{M_{\sigma}}{\sigma}|B|\|w\|_{[t_0, t]},
	\end{align}	and
	\begin{align}
		\nonumber\|w\|_{[t_0, t]}&\le e^{D\left(\nu+1\right)}(1+\varepsilon) \sup _{t_0\leqslant s \leqslant t}\left(|d(s)| e^{\delta (s-t_0)}\right) \label{norm_w_0_t} \\
		&+e^{D}\left\|w(t_0)\right\|_{\infty}.
	\end{align}	
	Let us next estimate the term $\displaystyle\sup_{t_0 \leqslant s\leqslant t}\left(|d(s)| e^{\delta (s-t_0)}\right).$ 	From \eqref{d1} we get for $t_0<t\le t_0+T$
	\begin{align}
		\nonumber		|d|& \leqslant\left|K e^{A D}\right| \mu\left|q_1\left(\frac{X}{\mu}\right)-\frac{X}{\mu}\right|\\
		\nonumber		&+D \sup _{0 \leqslant y \leqslant D}|Ke^{A(D-y)}B| \mu\left\|q_2\left(\frac{u}{\mu}\right)-\frac{u}{\mu}\right\|_{\infty} \\
		&\le  M_{3} \mu \left\|q_1\left(\frac{X}{\mu}\right)-\frac{X}{\mu}, q_2\left(\frac{u}{\mu}\right)-\frac{u}{\mu}\right\|\label{d},
	\end{align} with $M_{3}$ defined in \eqref{M3} and $u$ given in terms of $(X,w)$ by the inverse backstepping transformation \eqref{backstepping_inverse_transformation}. Provided that 
	\begin{align}
		\label{condition}
		\Omega \frac{M_{2}}{(1+M_0)^2} M\mu \leq|X|+\|w\|_{\infty} \leqslant M_{2} M\mu,
	\end{align}
	thanks to the property $\rm P1$ of the quantizer, the left-hand side of bound \eqref{equivalence}, and the definition \eqref{Omega}, we obtain \begin{align}
		\nonumber	|d|& \le  M_{3}\Delta\mu\\
		\nonumber	&\leqslant \frac{(1+M_0)^2 M_{3}\Delta}{\Omega M_2M} \left(|X|+\|w\|_{\infty}\right) \\
		& \leqslant \frac{1}{1+\lambda}\left(|X|+\|w\|_{\infty}\right). 
	\end{align}	Therefore, as long as the solutions satisfy \eqref{condition} we get
	\begin{equation}
		\sup_{t_0 \leqslant s \leqslant t}\left(|d(s)| e^{\delta (s-t_0)}\right) \leqslant \frac{1}{1+\lambda}\|w\|_{[t_0, t]}+\frac{1}{1+\lambda}|X|_{[t_0, t]}.
	\end{equation}
	Hence, using \eqref{norm_w_0_t} and \eqref{phi and psy}, we obtain
	\begin{align}
		\nonumber	\|w\|_{[t_0, t]}& \leqslant e^{D}\left\|w(t_0)\right\|_{\infty}+\frac{1+\varepsilon}{1+\lambda} e^{D\left(\nu+1\right)}\|w\|_{[t_0, t]}\\
		&+\frac{1+\varepsilon}{1+\lambda} e^{D\left(\nu+1\right)}|X|_{[t_0, t]}, 
	\end{align}
	and hence,
	\begin{align}
	 \|w\|_{[t_0, t]}\leqslant\left(1-\phi\right)^{-1}e^D\left\|w(t_0)\right\|_{\infty}+\left(1-\phi\right)^{-1}\phi |X|_{[t_0, t]} ,\label{normw}
	\end{align}
	with $\phi=\frac{1+\varepsilon}{1+\lambda}e^{D\left(\nu+1\right)}<1.$
	Combining \eqref{norm_x_w} and \eqref{normw} one gets
	\begin{align}
		\nonumber	\|w\|_{[t_0,t]} &\leqslant\left(1-\phi\right)^{-1}\left(1-\varphi_1\right)^{-1}e^D\left\|w(t_0)\right\|_{\infty}\\
		\label{normw1}	&+\left(1-\phi\right)^{-1} \phi M_{\sigma}\left(1-\varphi_1\right)^{-1}|X(t_0)|, 
	\end{align} where we use the fact that from \eqref{phi and psy} it follows that $
	\varphi_1=(1+\varepsilon)(1-\phi)^{-1}\phi \frac{M_{\sigma}}{\sigma}|B|<1.$
	Combining \eqref{norm_x_w} and \eqref{normw} we also arrive at 
	\begin{align}
		\nonumber	 |X|_{[t_0, t]}&\le  (1+\varepsilon) \frac{M_{\sigma}}{\sigma}|B|\left(1-\phi\right)^{-1}e^D\left\|w(t_0)\right\|_{\infty} \\
		& + M_{\sigma}|X(t_0)|+(1+\varepsilon) \frac{M_{\sigma}}{\sigma} |B|\left(1-\phi\right)^{-1} \phi|X|_{[t_0, t]},
	\end{align} 
	that is
	\begin{align}
		\nonumber|X|_{[t_0, t]} &\leqslant(1-\varphi_1)^{-1}M_{\sigma}|X(t_0)|\\
		&+(1+\varepsilon) \left(1-\phi\right)^{-1}(1-\varphi_1)^{-1}e^D\frac{M_{\sigma}}{\sigma}|B|\left\|w(t_0)\right\|_{\infty}. \label{normX1}
	\end{align}  Therefore, one obtains from \eqref{normw1}
	\begin{equation}
		\|w\|_{[t_0, t]} \leq C_{0}\left(|X(t_0)|+\left\|w(t_0)\right\|_{\infty}\right),
	\end{equation} with
	\begin{multline}
		C_{0}=\max \left\{\left(1-\phi\right)^{-1}\left(1-\varphi_{1}\right)^{-1}e^D ;\right. \\\left.\left(1-\phi\right)^{-1} \left(1-\varphi_{1}\right)^{-1}\phi M_{\sigma}\right\},
	\end{multline}
	and from \eqref{normX1}
	\begin{equation}
		|X|_{[t_0,t]} \leq C_{1}\left(|X(t_0)|+\left\|w(t_0)\right\|_{\infty}\right), 
	\end{equation} with 
	\begin{align}
		\nonumber C_{1}&=\max \left\{(1-\varphi_1)^{-1}M_{\sigma}; \right. \\ &\left.(1+\varepsilon) \left(1-\phi\right)^{-1}(1-\varphi_1)^{-1}e^D\frac{M_{\sigma}}{\sigma}|B|\right\}.
	\end{align}
	Therefore, setting $M_0=C_{0}+ C_{1} $ we get
	\begin{equation} \label{eqXwM0}
		|X|_{[t_0, t]}+	\|w\|_{[t_0, t]} \leqslant M_0\left(|X(t_0)|+\left\|w(t_0)\right\|_{\infty}\right), 
	\end{equation}
	and using the definitions \eqref{defnormw}, \eqref{defnormX} we obtain
	\begin{equation}\label{norm_X_u1}
		|X(t)|+\|w(t)\|_{\infty}\leqslant M_0e^{-\delta (t-t_0)}\left(|X(t_0)|+\left\|w(t_0)\right\|_{\infty}\right). 
	\end{equation}
	For $t_{0}<t \leq t_{0} +T $, using relation \eqref{hyplemma2}, the fact that $e^{-\delta\left(t-t_{0}\right)}\le 1$,  and $\frac{M_0}{1+M_0}<1$ one has 
	\begin{align}
		|X(t)|+\|w(t)\|_{\infty} 
		\label{forMaxTerm}	& \le M_2M\mu. 
	\end{align} Moreover, at the time instant $t_0+T$, thanks to the relation \eqref{hyplemma2} and the definition \eqref{T} of $T,$ one obtains from \eqref{norm_X_u1} that \begin{align}
		\left|X\left(t_{0}+T\right)\right|+\left\|w\left(t_{0}+T\right)\right\|_{\infty} \label{norm_X_u_T}&\leqslant \Omega \frac{M_2}{1+M_0}M\mu. 
	\end{align} Note that relations \eqref{eqXwM0} and \eqref{norm_X_u_T} are established provided that $\Omega \frac{M_{2}}{(1+M_0)^2} M\mu \leq|X|+\|w\|_{\infty}$. If there exists a time $t_0^*$ such that $t_0\le t_0^*\le t_0+T,$ at which the solutions satisfy \begin{align}\label{eq:t0*}
		\left|X\left(t_{0}^{*}\right)\right|+\left\|w\left(t_{0}^{*}\right)\right\|_{\infty} & \leqslant \Omega
		\frac{M_{2}}{(1+M_0)^2} M\mu,
	\end{align} then they also satisfy for $t_{0}^{*}\le t\le t_{0}+T$
	\begin{align}\label{eq:t01}
		|X(t)|+\|w(t)\|_{\infty}\leqslant  \Omega
		\frac{M_{2}}{1+M_0} M\mu, 
	\end{align} and therefore, combining this estimate with \eqref{norm_X_u1} we obtain the bound \eqref{norm_X_w}. To see this, note that once the solutions may enter again the region where \eqref{condition} holds (if this happens) then the previous analysis becomes legitimate. In particular, estimate (76) is activated, which implies that the solutions remain in the region where \eqref{eq:t01} holds.
\end{proof} We note here that the proof strategy of Lemma~\ref{Lemma2} relies on the objective to establish an ultimate boundedness property for the target system. Although, in \cite{bekiaris2020hybrid,liberzon2003hybrid}, a Lyapunov-like analysis is employed during the zooming-in stage, facilitating the attainment of an ultimate boundedness property (through the invariance of considered regions), our approach here relies on a small-gain ISS argument. Therefore, to establish an ultimate bound estimate, our analysis relies on a choice of the zooming in parameter $\Omega$ and the dwell-time $T$ (where $T$ denotes the time instant at which solutions enter a desired region), both dependent on the overshoot $M_0$.

We are now ready to prove Theorem \ref{Theorem1}.

{\em Proof of Theorem~\ref{Theorem1}:} 
The inequality 	$|X(t_0)|+\|u(t_{0})\|_{\infty} \leq M \overline{M}\mu$ in Lemma~\ref{Lemma1} holds with constant $\mu=\mu\left(t_{0}\right).$ Therefore, using \eqref{equivalence} and the definition \eqref{M} of $\overline{M}$, the inequality \eqref{hyplemma2} holds. Then applying Lemma~\ref{Lemma2} where $\mu$ is updated according to \eqref{switching_parameter} the inequality \eqref{normXu1} holds  with $\mu=\mu(t)=\mu(t_0+T)$. Thus, relation \eqref{normXu1} implies that \eqref{hyplemma2} holds but with $t_0\to t_0+T$ and $\mu=\mu(t_0+2T)=\Omega\mu(t_0+T)=\Omega\mu(t_0).$
Then by applying again Lemma~\ref{Lemma2}, we have for $t_{0}+T<t\le t_{0}+2 T$ and $\mu=\mu(t_0+2T)$
\begin{align}
	\label{eq2T}|X(t_0+2T)|+\|w(t_0+2T)\|_{\infty} \leq \Omega^2 \frac{M_{2}}{1+M_0} M\mu(t_0).
\end{align}Using the estimate \eqref{norm_X_w} in Lemma~\ref{Lemma2}, we have for\\ $t_{0}+T<t\le t_{0}+2 T$
\begin{align}
	\nonumber|X(t)|+\|w(t)\|_{\infty}&\leqslant \max \left\{M_{0} e^{-\delta (t-t_0-T)}\left(\left|X(t_0+T)\right|\right.\right.\\
	&\left.\left. +\left\|w(t_0+T)\right\|_{\infty}\right), \Omega \frac{M_{2}}{1+M_0} M\mu(t) \right\}.\label{eqmax}
\end{align}
Therefore, since in \eqref{switching_parameter} for $t_{0}+T<t\le t_{0}+2 T$, $\mu(t)=\Omega\mu(t_0),$ using \eqref{normXu1} we obtain for $t_{0}+T<t\le t_{0}+2 T$
\begin{align}
	|X(t)|+\|w(t)\|_{\infty} \leqslant  \Omega M_{2} M\mu(t_0).\label{eqmax2T}
\end{align}  Repeating this procedure, we arrive at
\begin{align}
	|X(t)|+\|w(t)\|_{\infty}\leqslant \Omega^{i-1}M_2 M\mu(t_0),\label{bound_Norm_X_u_i2}
\end{align} for all $ t_{0}+(i-1) T<t\le  t_{0}+i T.$ 
Therefore for $t_{0}+(i-1) T<t\le  t_{0}+i T$, $i=1,2,\dots,$ we get
\begin{align}
	&|X(t)|+\|w(t)\|_{\infty} 	\label{bound_Norm_X_w_i2}\le \Omega^{\left(\frac{t-t_0}{T}\right)} \frac{M_2M}{\Omega}\mu(t_0).
\end{align} From the definition of $\mu$ in \eqref{switching_parameter} one has
\begin{align}
	\mu(t_0)&\leq \overline{M}_1 e^{2|A| \tau } e^{2|A| t_{0} }\mu_{0},
\end{align} thus, for $t\ge t_0$ it follows from \eqref{bound_Norm_X_w_i2} that 
\begin{align}
	&|X(t)| +\|w(t)\|_{\infty} \leq \mu_{0}\overline{M}_1\frac{M_2M}{\Omega}e^{2|A| \tau} e^{2|A| t_{0}}  e^{\left(t-t_{0}\right) \frac{\ln \Omega}{T}}.
\end{align}
Using estimate \eqref{normX0t0} and inequality \eqref{equivalence} we obtain
\begin{align}
	|X(t)|+\|w(t)\|_{\infty} \leqslant \overline{M}_1M_1e^{|A| t}\left(\left|X_{0}\right|+\left\|u_{0}\right\|_{\infty}\right),
\end{align} for $0 \leqslant t \leqslant t_{0}.$
Combining these two last estimates and the inequality \eqref{equivalence} we get for all $t\ge 0$
\begin{align}
	\nonumber	|X(t)|+\|u(t)\|_{\infty} 
	&\leqslant \max \left\{e^{|A| t_{0}},  |X_{0}|+\|u_0\|_{\infty}\right\} \overline{M}_{2} \\
	&\times e^{|A| t_{0}} e^{-\frac{\ln \Omega}{T} t_{0}}e^{\frac{\ln \Omega}{T} t},	\label{normXuf}
\end{align} where 
\begin{align}
	\overline{M}_{2}=\frac{\overline{M}_1}{M_2}\max \left\{\frac{M_2M}{\Omega} e^{2|A| \tau} \mu_{0}, M_1\right\}.
\end{align}
From \eqref{t0} we have \begin{equation}
	t_{0} \leqslant \frac{1}{|A|} \ln \left[\overline{M}_{3}\left(\left|X_{0}\right|+\left\|u_{0}\right\|_{\infty}\right)\right],
\end{equation} with \begin{equation}
	\overline{M}_{3}=\frac{1}{\mu_0(M \overline{M}-2 \Delta)}.
\end{equation} Thus,
\begin{equation}
	\label{expAt0}e^{|A| t_{0}} \leq \overline{M}_{3}\left(\left|X_{0}\right|+\left\|u_{0}\right\|_{\infty}\right).
\end{equation}  Moreover,
\begin{align}
	-\frac{\ln\Omega}{T} t_{0} &\leqslant-\frac{\ln \Omega}{T}\frac{1}{|A|} \ln\left[\overline{M}_{3} \left(|X_{0}|+\|u_{0}\|_{\infty}\right) \right],
\end{align}
\text{ and thus, }
\begin{align}
	\nonumber e^{-\frac{\ln \Omega}{T} t_{0}} &\leqslant e^{\ln \overline{M}_{3}\left(\left|X_{0}\right|+\left\|u_{0}\right\|_{\infty}\right)^{-\frac{\ln \Omega}{T} \times \frac{1}{|A|}}} \\
	& \leqslant \overline{M}_{3}^{-\frac{\ln \Omega}{T} \times \frac{1}{|A|}}\left(\left.\right|{X_{0}}|+ \|{u_{0}}\|_{\infty}\right)^{-\frac{\ln \Omega}{T} \times \frac{1}{|A|}}\label{expLnOmegaT}.
\end{align}
From \eqref{expAt0} and \eqref{expLnOmegaT} one obtains
\begin{align}
	\nonumber	 \max \left\{e^{|A| t_{0}},\left|X_{0}\right|+\left\|u_{0}\right\|_{\infty}\right\}& \le \left(\left|X_{0}\right|+\left\|u_{0}\right\|_{\infty}\right)\\&
	\times	\max \left\{\overline{M}_{3}, 1\right\}, \label{max1}\\
	\nonumber e^{|A| t_{0}} e^{-\frac{\ln\Omega}{T} t_{0}}&\le  \overline{M}_{3}^{\left(1-\frac{\ln \Omega}{T} \times \frac{1}{|A|}\right)}\\ 
	\hspace{-2cm}	\label{expAexp}&\times \left(\left|X_{0}\right|+\|u_{0}\|_{\infty}\right)^{\left(1-\frac{\ln \Omega}{T|A|}  \right)}.
\end{align}
Therefore, from  \eqref{max1} and \eqref{expAexp} we arrive at
\begin{align}
	\nonumber|X(t)|+\|u(t)\|_{\infty}&\leq \max \left\{\overline{M}_{3}, 1\right\}\overline{M}_{2} \overline{M}_{3}^{\left(1-\frac{\ln \Omega}{T} \frac{1}{|A|}\right)}\\
	&\times\left(\left|X_{0}\right|+\|u_{0}\|_{\infty}\right)^{\left(2-\frac{\ln \Omega}{T} \frac{1}{|A|}\right)}  e^{\frac{\ln \Omega}{T} t}, 
\end{align} which gives \eqref{stability_result_u}.

We now prove well-posedness. In the interval $\left[0, t_0\right]$, where there is no control, the system described by \eqref{pde_representation0}--\eqref{pde_representation3}. The existence and uniqueness of solutions within this interval are ensured by the explicit solution to the ODE subsystem \eqref{pde_representation0} and the transport subsystem \eqref{pde_representation3}, thanks to the variation of constants formula and the characteristics method, respectively. These solutions depend only on $X_{0} \in \mathbb{R}^{n}$ and $u_{0} \in \mathcal{C}_{lpw}([0, D],\mathbb{R})$ and one has $X(t)\in AC\left([0,t_0], \mathbb{R}^{n}\right)$ and  $u\in \mathcal{C}_{lpw}([0,D]\times[0,t_0]), \mathbb{R})$. 
For $t>t_0$, the system described by \eqref{pde_representation}--\eqref{pde_representation1}, along with the quantized controller $U$, defined in \eqref{control_quantizer}, satisfies the assumptions outlined in \cite[Theorem 8.1]{karafyllis2021input}, with $F(X,u)=A X+B u(0)$ and $\varphi(\mu,u,X)=U(\mu,u,X)$. In particular, $U$ defined in \eqref{control_quantizer}, \eqref{predictor_quantizer} is locally Lipschitz in $(X, u)$, given the local Lipschitzness assumption of $q_1$ and $q_2$. Therefore, the initial conditions for each interval $I_i=\left[t_{0}+(i-1) T, t_{0}+i T\right],$ where $i=1,2\dots,$ satisfy  $X\left(t_{0}+(i-1) T\right)\in  \mathbb{R}^{n}$, $u\left(x,t_{0}+(i-1) T\right)\in \mathcal{C}_{lpw}([0, D], \mathbb{R}),$ and they are bounded due to \eqref{bound_in_time_t0} for $i=1$ and \eqref{norm_X_w} for $i \geqslant 2$, respectively. Then, the system \eqref{pde_representation}--\eqref{pde_representation1} with \eqref{control_quantizer}, given the initial conditions $X\left(t_{0}+(i-1) T\right)$ and $u(\cdot,t_0+(i-1)T)\in \mathcal{C}_{lpw}([0, D], \mathbb{R})$, where $i=1,2, \cdots$, admits a unique solution such that $X(t)\in AC\left(I_{i+1}, \mathbb{R}^{n}\right)$ and for each $t,$ $u(\cdot,t)\in \mathcal{C}_{lpw}([0, D], \mathbb{R})$, while for each $x,$ $u(x,\cdot)\in \mathcal{C}_{lpw}(I_{i+1}, \mathbb{R}).$ (This regularity of the solution is also obtained in \cite{espitia2017stabilization,KKnonlocal} in the context of transport PDE systems subject to sampling-data and quantization.) Therefore, using a proof by induction, we obtain the existence and uniqueness of a solution such that $X(t)$ is absolutely continuous in $[0,+\infty)$, while $u(\cdot,t)\in \mathcal{C}_{lpw}([0, D], \mathbb{R})$ and  $u(x,\cdot)\in \mathcal{C}_{lpw}(\mathbb{R}_{+}, \mathbb{R})$ for each $t$ and $x$, respectively.
{\hfill $\Box$ }

\section{Extension to Input Quantization}\label{inputquantization}
When the control input is subjected to quantization, yet measurements of the actuator/ODE states are available, modifications to the switched predictor-feedback law are required. Specifically, the adaptation entails
 \begin{equation}\label{control_quantizerinput}
		U(t)=\left\{\begin{array}{ll}0, & 0 \leq t \leq \bar{t}_{0} \\  \bar{q}_{\mu}\left(U_{\rm nom}(t)\right), & t>\bar{t}_{0}
			\end{array}\right.,
	\end{equation} where $U_{\rm nom}(t)$ is given in \eqref{nominalU}.
The quantizer is a function $\bar{q}_{\mu}:\mathbb{R}\to \mathbb{R},$ defined by $\bar{q}_{\mu}(\bar{U})=\mu \bar{q}\left(\tfrac{\bar{U}}{\mu}\right)$, satisfying the following properties

$\bar{\rm P}1$: If $|\bar{U}| \leq M$, then $|\bar{q}(\bar{U})-\bar{U}| \leq \Delta$,\\
$\bar{\rm P}2$: If $|\bar{U}|>M$, then $|\bar{q}(\bar{U})|>M-\Delta$,\\
$\bar{\rm P}3$: If $|\bar{U}| \leq \hat{M}$, then $\bar{q}(\bar{U})=0$.

 The tunable parameter $\mu$ is selected as 
	\begin{equation}\label{switching_parameterinput}
		\mu(t)= \begin{cases}\overline{M}_1 \mathrm{e}^{2|A| j \tau} \mu_{0}, & (j-1) \tau \leq t \leq j \tau, \\ & 1 \leq j \leq\left\lfloor\frac{\bar{t}_{0}}{\tau}\right\rfloor+1,\\ \mu\left(\bar{t}_{0}\right), & t \in\left(\bar{t}_{0}, \bar{t}_{0}+T\right], \\ \Omega \mu\left(\bar{t}_{0}+(i-1) T\right), & t \in\left(\bar{t}_{0}+(i-1) T,\right. \\ & \left.\bar{t}_{0}+i T\right], \quad i=2,3, \ldots\end{cases},
	\end{equation} for some fixed, yet arbitrary, $\tau, \mu_0>0$, with $\bar{t}_{0}$ being the first time instant at which the following holds 
	\begin{equation}\label{bound_in_time_t0input}
		|X(\bar{t}_{0})|+\|u(\bar{t}_0)\|_{\infty} \leq  \dfrac{M \overline{M}}{M_3}\mu(\bar{t}_0).
	\end{equation}	
Note that this event can be detected as measurements of the actuator/ODE states are available. The parameters involved in \eqref{switching_parameterinput}, \eqref{bound_in_time_t0input} are defined in \eqref{M3}--\eqref{T}.
	\begin{Theo}\label{Theorem2}
		Consider the closed-loop system consisting of the plant \eqref{pde_representation}--\eqref{pde_representation1} and the switched predictor-feedback law \eqref{control_quantizerinput},    \eqref{switching_parameterinput} with \eqref{nominalU}. Let the pair $(A, B)$ be stabilizable. If $\Delta$ and $M$ satisfy  \begin{equation}\label{conditionMDeltainput}
				\frac{\Delta}{M}<\frac{M_2}{M_3(1+\lambda)(1+M_0)^2},
			\end{equation}	then for all $X_{0} \in \mathbb{R}^{n}$; $u_{0} \in \mathcal{C}_{lpw}([0, D], \mathbb{R})$, there exists a unique solution such that $X(t) \in AC\left(\mathbb{R}_{+}, \mathbb{R}^{n}\right)$ and for each $t \in \mathbb{R}_{+}$ $u(\cdot, t) \in \mathcal{C}_{lpw}([0,D], \mathbb{R})$ and for each $x \in[0,D]$ $u(x,\cdot) \in \mathcal{C}_{\text {lpw }}\left(\mathbb{R}_{+}, \mathbb{R}\right)$, which satisfies
		\begin{align}
				&|X(t)|+\left\|u(t)\right\|_{\infty}
				\label{stability_result_uinput} \leq  \bar{\gamma}\left(|X_0|+\left\|u_{0}\right\|_{\infty}\right)^{\left(2-\frac{\ln \Omega}{T} \frac{1}{|A|}\right)} \mathrm{e}^{\frac{\ln \Omega}{T}t},
			\end{align}	where		
			\begin{align}
			\nonumber	\bar{\gamma} &=\frac{\overline{M}_1}{M_2}\max \left\{\frac{M_{2}M}{\Omega M_3} e^{2|A| \tau} \mu_{0}, M_1\right\} \max \left\{\frac{M_3}{\mu_0M \overline{M}}, 1\right\}\\
			&
			\times \left(\frac{M_3}{\mu_0M \overline{M}}\right)^{\left(1-\frac{\ln \Omega}{T} \frac{1}{|A|}\right)}.
		\end{align} 
	\end{Theo}	
	The proof of Theorem~\ref{Theorem2} relies also on two lemmas, whose proofs utilize analogous arguments to those applied to the case of measurements' quantization.
	\vspace{.2cm}		
	\begin{Lem}\label{Lemma3}
	There exists a time $\bar{t}_0$ satisfying 
		\begin{equation}\label{t0input}
				\bar{t}_{0} \leqslant \frac{1}{|A|} \ln\left(\dfrac{\frac{M_3}{\mu_{0}}\left(\left|X_{0}\right|+\left\|u_{0}\right\|_{\infty}\right)}{M \overline{M}}\right),
		\end{equation} 
		such that \eqref{bound_in_time_t0input} holds.		
	\end{Lem}
	\begin{proof}	
		For all $0\le t\le \bar{t}_{0}$, the system is given by \eqref{pde_representation}--\eqref{pde_representation1} with $U(t)=0$.	Using the method of characteristics and the variation of constants formula for $ 0\le t\le \bar{t}_{0}$ we obtain exactly as in the proof
		of Lemma \ref{Lemma1}
		\begin{equation}\label{normX0t0input}
			|X(t)|+\|u(t)\|_{\infty} \leqslant \overline{M}_1e^{|A| t}\left(\left|X_{0}\right|+\left\|u_{0}\right\|_{\infty}\right).
		\end{equation} Choosing the switching signal $\mu$ according to \eqref{switching_parameterinput}, one has the existence of a time $\bar{t}_{0}$ verifying \eqref{t0input}  such that the relation \eqref{bound_in_time_t0input} holds.
    \end{proof}
	\begin{Lem}\label{Lemma4}
		Choose $K$ such that $A+BK$ is Hurwitz and let $\sigma,M_{\sigma}>0$ be such that \eqref{expABK} holds. Select $\lambda$ large enough in such a way that the small-gain condition \eqref{small_gain} holds. Then the solutions to the target system \eqref{targetsystem_without_quantizer}--\eqref{targetsystem_without_quantizer3} with the quantized controller \eqref{nominalU}, \eqref{control_quantizerinput}, \eqref{switching_parameterinput}, which verify, for fixed $\mu$,
		\begin{equation}\label{hyplemma4input}
				|X(\bar{t}_{0})|+\|w(\bar{t}_{0})\|_{\infty}\le \frac{M_{2}M\mu}{(1+M_0)M_3},
			\end{equation} they satisfy for $\bar{t}_0<t\le \bar{t}_0+T$
		\begin{align}
				\nonumber& |X(t)|+\|w(t)\|_{\infty}\leqslant \max\left\{ M_{0} e^{-\delta (t-\bar{t}_{0})}\left(\left|X(\bar{t}_{0})\right| \right.\right.\\
				&\left. +\left\|w(\bar{t}_{0})\right\|_{\infty}\right),\left.  \frac{\Omega M_{2}M\mu }{(1+M_0)M_3}  \right\}.\label{norm_X_winput}
			\end{align} In particular, the following holds
		\begin{align}\label{normXu1input}
				|X(\bar{t}_{0}+T)|+\|w(\bar{t}_{0}+T)\|_{\infty} \leq  \frac{\Omega M_{2}M\mu}{(1+M_0)M_3}.
			\end{align}
	\end{Lem}
\begin{proof}
For $\bar{t}_0< t\le \bar{t}_0+T,$ the system is defined by \eqref{pde_representation}--\eqref{pde_representation1} under the switched predictor-feedback law \eqref{nominalU}, \eqref{control_quantizerinput}, \eqref{switching_parameterinput}. Using the same strategy, as in the proof of  Lemma~\ref{Lemma2}, i.e., combining the variation of constants  formula, the ISS estimate in sup-norm in \cite[estimate (2.23)]{KKnonlocal} (see also \cite[estimate (3.2.11)]{karafyllis2021input}), and the fading memory lemma  \cite[Lemma 7.1]{karafyllis2021input}, for every $\nu, \varepsilon>0$ satisfying \eqref{small_gain2}, there exists $\delta \in (0, \min\{\sigma,\nu\})$ such that, using the definitions \eqref{defnormw} and \eqref{defnormX}, the following inequalities hold
	\begin{align}\label{norm_x_winput}
		|X|_{[\bar{t}_0, t]} \leqslant M_{\sigma}|X(\bar{t}_0)|+(1+\varepsilon) \frac{M_{\sigma}}{\sigma}|B|\|w\|_{[\bar{t}_0, t]},
	\end{align}	and
	\begin{align}
		\nonumber\|w\|_{[\bar{t}_0, t]}&\le e^{D\left(\nu+1\right)}(1+\varepsilon) \sup _{\bar{t}_0\leqslant s \leqslant t}\left(\left|\bar{d}(s)\right| e^{\delta (s-\bar{t}_0)}\right) \label{norm_w_0_tinput} \\
		&+e^{D}\left\|w(\bar{t}_0)\right\|_{\infty},
	\end{align}	where
\begin{align}
	 \bar{d}(t)=U_{\rm nom}(t)-\mu(t)\bar{q}\left(\dfrac{U_{\rm nom}(t)}{\mu(t)}\right) ,	\label{d1input}
\end{align} with $ U_{\rm nom}$ and $\mu$ given in \eqref{nominalU} and \eqref{switching_parameterinput}, respectively. Next, let us proceed to approximate the term $\displaystyle\sup_{\bar{t}_0 \leqslant s\leqslant t}\left(\left|\bar{d}(s)\right| e^{\delta (s-\bar{t}_0)}\right)$ in \eqref{norm_w_0_tinput} (for fixed $\mu$).   
Provided that 
\begin{align}
	\label{conditioninput}
	\frac{\Omega}{(1+M_0)^2}\frac{M_{2}}{M_3} M\mu \leq|X|+\|w\|_{\infty} \leqslant \frac{M_{2}}{M_3} M\mu,
\end{align}
	using the property $\rm \bar{P}1$ of the quantizer, the left-hand side of bound \eqref{equivalence}, and the definition \eqref{Omega}, we obtain \begin{align}
		\nonumber	\left|\bar{d}\right|&=\mu\left|\bar{q}\left(\dfrac{U_{\rm nom}(t)}{\mu}\right)-\dfrac{U_{\rm nom}(t)}{\mu}\right|  \\
	\nonumber	& \le    \Delta\mu          \\
		\nonumber	&\leqslant \frac{(1+M_0)^2 M_{3}\Delta}{\Omega M_2M} \left(|X|+\|w\|_{\infty}\right) \\
		& \leqslant \frac{1}{1+\lambda}\left(|X|+\|w\|_{\infty}\right). 
	\end{align}	Therefore, as long as the solutions satisfy \eqref{conditioninput} we get
	\begin{equation}
		\sup_{\bar{t}_0 \leqslant s \leqslant t}\left(\left|\bar{d}(s)\right| e^{\delta (s-\bar{t}_0)}\right) \leqslant \frac{1}{1+\lambda}\|w\|_{[\bar{t}_0, t]}+\frac{1}{1+\lambda}|X|_{[\bar{t}_0, t]}.
	\end{equation}
	Hence, using \eqref{norm_w_0_tinput} and the fact that \eqref{small_gain1} holds we obtain
	\begin{align}
		\|w\|_{[\bar{t}_0, t]}\leqslant\left(1-\phi\right)^{-1}e^D\left\|w(\bar{t}_0)\right\|_{\infty}+\left(1-\phi\right)^{-1}\phi |X|_{[\bar{t}_0, t]} ,\label{normwinput}
	\end{align}
	with $\phi=\frac{1+\varepsilon}{1+\lambda}e^{D\left(\nu+1\right)}<1.$
	Combining the inequalities \eqref{norm_x_winput} and \eqref{normwinput}, thanks to the definitions \eqref{defnormw}, \eqref{defnormX}, and to the small-gain condition \eqref{small_gain}, repeating the respective arguments from the proof of Lemma \ref{Lemma2}, we arrive at
	\begin{equation}\label{norm_X_u1input}
		|X(t)|+\|w(t)\|_{\infty}\leqslant M_0e^{-\delta (t-\bar{t}_0)}\left(|X(\bar{t}_0)|+\left\|w(\bar{t}_0)\right\|_{\infty}\right). 
	\end{equation}
	For $\bar{t}_0<t \leq \bar{t}_0 +T $, using relation \eqref{hyplemma4input}, the fact that $e^{-\delta\left(t-\bar{t}_0\right)}\le 1$,  and $\frac{M_0}{1+M_0}<1$ one has that
	\begin{align}
		|X(t)|+\|w(t)\|_{\infty} 
		\label{forMaxTerminput}	& \le \frac{M_2}{M_3}M\mu, 
	\end{align} which makes estimate \eqref{norm_X_u1input} legitimate. Moreover, at the time instant $\bar{t}_0+T$, thanks to the relation \eqref{hyplemma4input} and the definition \eqref{T} of $T,$ one obtains from \eqref{norm_X_u1input} that \begin{align}
		\left|X\left(\bar{t}_0+T\right)\right|+\left\|w\left(\bar{t}_0+T\right)\right\|_{\infty} \label{norm_X_u_Tinput}&\leqslant \Omega \frac{M_2}{(1+M_0)M_3}M\mu, 
	\end{align}  and hence, bound \eqref{norm_X_winput} is obtained using \eqref{normwinput}.
The rest of the proof utilizes the same reasoning as the proof of Lemma~\ref{Lemma2}.
\end{proof}
{\em Proof of Theorem~\ref{Theorem2}:} The method used to prove Theorem~\ref{Theorem2} closely follows the method employed for the corresponding part of the proof of Theorem~\ref{Theorem1}, utilizing Lemmas~\ref{Lemma3} and~\ref{Lemma4}, in correspondence to Lemmas \ref{Lemma1} and \ref{Lemma2}.
\section{Simulation Results}\label{simulation}
To illustrate Theorem~\ref{Theorem1}, we consider system \eqref{pde_representation}--\eqref{pde_representation1} with
\begin{equation}
	A = \begin{bmatrix}
		-1 & 1 \\
		0 & 1
	\end{bmatrix}; \quad B = \begin{bmatrix}
		0 \\
		1
	\end{bmatrix},
\end{equation} and initial conditions \(X_0 = (10\quad 0)^T\) and \(u_0(x) = 10\), for all \(x \in [0, D]\). For \(K = [0 \quad -3]\), the matrix \(A + BK\) is Hurwitz so that for \(\sigma = 1\) and \(M_{\sigma} = 0.5\), the relation \eqref{expABK} is satisfied. One has \(M_1 = 4.5\) and \(M_2 = 0.2\). The constant parameters \eqref{M3}--\eqref{T} are given by \(\overline{M} = 0.6\), \(\overline{M}_1 = 2\), \(\Omega = 0.63\), and \(T = 2\). Choosing \(\lambda = 12\), the small-gain condition \eqref{small_gain} is verified. The quantizer is defined component-wise for each \(x \in [0, D]\) as
\begin{equation}\label{quantizer_matrix}
	q_{\mu}\left(\begin{bmatrix}X_1\\X_2
	\end{bmatrix},u\right)=\left(\begin{bmatrix}
	\mu q\left(\dfrac{X_1}{\mu}\right) \quad \mu q\left(\dfrac{X_2}{\mu}\right)
\end{bmatrix}^T,\mu q\left(\dfrac{u}{\mu}\right) \right),
\end{equation}
with \begin{align}
	q\left(\frac{u(x)}{\mu}\right)=\left\{\begin{array}{rr}
		M, & \frac{u(x)}{\mu}>M \\
		-M, & \frac{u(x)}{\mu}<-M \\
		\Delta\left\lfloor\frac{u(x)}{\mu \Delta}+\frac{1}{2}\right\rfloor, & -M \leq \frac{u(x)}{\mu} \leq M
	\end{array},\right.\label{quantizer_simus}
\end{align}
where $M=2$ and  $\Delta=\dfrac{M}{100}$, and the switching signal $\mu$ is updated according to \eqref{switching_parameter}.

The relation $\left| \begin{bmatrix}
	\mu(t_0) q\left(\dfrac{X_1(t_0)}{\mu(t_0)}\right) \quad \mu(t_0) q\left(\dfrac{X_2(t_0)}{\mu(t_0)}\right)
\end{bmatrix} \right|+\left\|\mu(t_0)q\left(\dfrac{u(t_0)}{\mu(t_0)}\right) \right\|_{\infty} \leq (M\overline{M} - \Delta)\mu(t_0)$ holds for $t_0=0$ showing that event \eqref{first_time_t0} is detected with the initial conditions $(X_0, u_0)$. 
\begin{figure*}[t]
	\centering{
		\subfigure{\includegraphics[width=1\columnwidth]{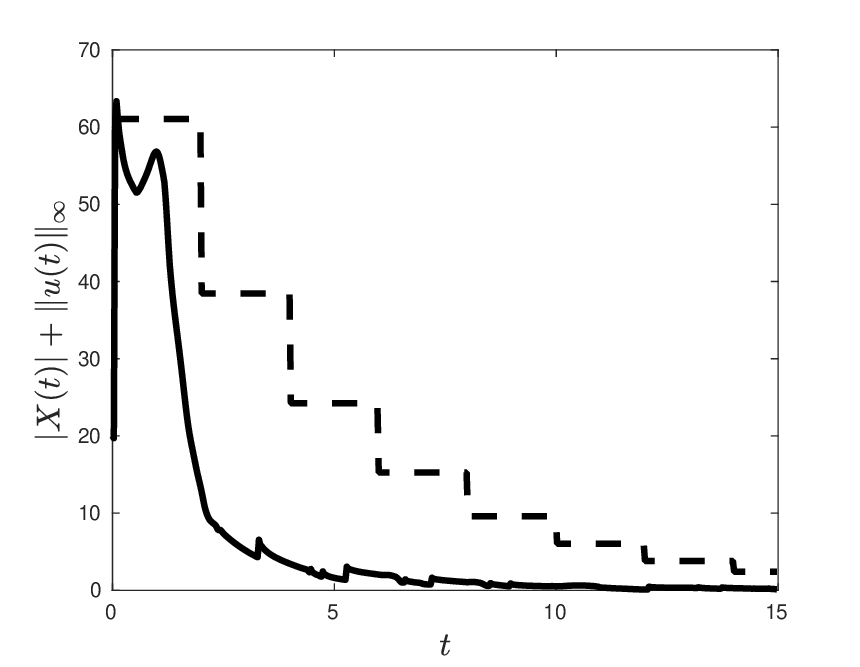} }
		\subfigure{\includegraphics[width=1\columnwidth]{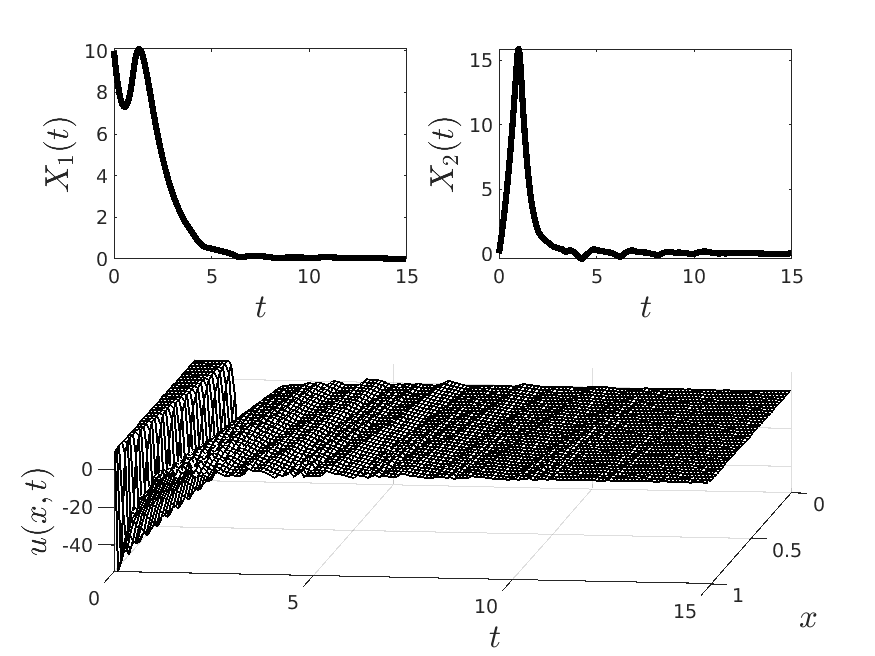} }
		\caption{Left: The norm $|X(t)| + \|u(t)\|_{\infty}$ of the closed-loop system \eqref{pde_representation}--\eqref{pde_representation1},       \eqref{quantizer_matrix}, \eqref{quantizer_simus},  for $D = 1$, $M = 2$, $\Delta =\dfrac{M}{100}$, under the predictor-feedback law \eqref{control_quantizer}--\eqref{switching_parameter}, \eqref{M3}--\eqref{T}, with parameters $\overline{M}=0.6,$  $\overline{M}_1=2,$  $\Omega=0.63,$ $T=2,$ and $\mu_0=1.$ The dashed line is the switching signal $\mu(t)M\overline{M}$.
			 Right: The respective states of the closed-loop system.}
		\label{dynamic_quantizer}
	}
\end{figure*}
We present in Figure~\ref{dynamic_quantizer} the norm of the state $(X,u)$ of
the closed-loop system along with the switching signal $\mu(t)M\overline{M}$, as well as the response of the ODE/actuator states. Specifically, for $x=D$, we display the control input
signal. The response $u$ of the closed-loop system is computed numerically using a Lax-Friedrichs scheme with time
and spatial discretization steps set to $0.01$ and $0.02$, respectively. The integral involved in the backstepping controller \eqref{control_quantizer} is numerically computed using the trapezoidal
rule. 
\begin{figure*}[t]
	\centering{
		\subfigure{\includegraphics[width=1\columnwidth]{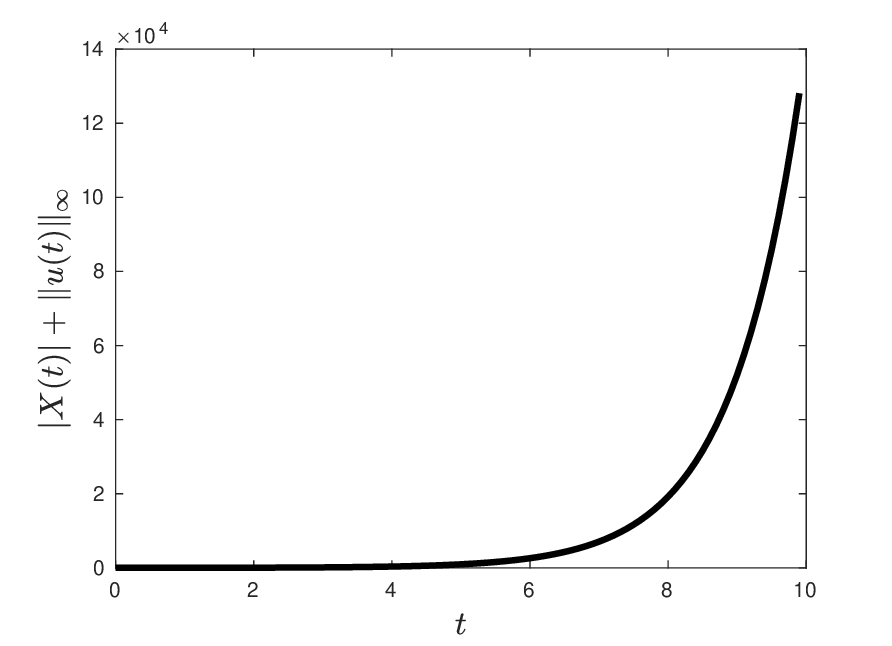} }
		\subfigure{\includegraphics[width=1\columnwidth]{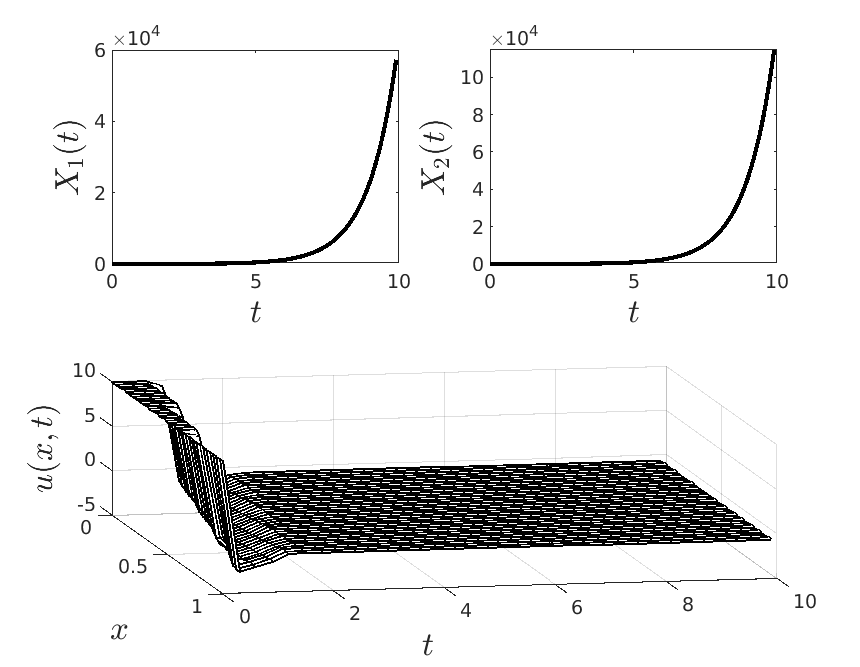} }
		\caption{Left: The norm $|X(t)| + \|u(t)\|_{\infty}$ of the closed-loop system \eqref{pde_representation}--\eqref{pde_representation1}, for $D = 1$, $M = 2$, $\Delta =\dfrac{M}{100}$, under the nominal feedback law
			$U(t)=\mu K\displaystyle\int_{0}^{D}e^{A(D-y)}Bq\left(\dfrac{u(y)}{\mu}\right)dy+\mu K e^{AD}\begin{bmatrix}
				q\left(\dfrac{X_1}{\mu}\right) \quad q\left(\dfrac{X_2}{\mu}\right)
			\end{bmatrix}^T,$
			for fixed $\mu = 0.1$ and $q$ defined in \eqref{quantizer_matrix}, \eqref{quantizer_simus}, with parameters $\overline{M}=0.6$, $\overline{M}_1 = 2$, $\Omega=0.63$, $T = 2$, and $\mu_0 = 1$. Right: The respective states of the closed-loop system.}
		\label{fixed_zoom_variable_small}
	}
\end{figure*}

To further elucidate the importance of the proposed
control design methodology, we present in Figures~\ref{fixed_zoom_variable_small} and \ref{fixed_zoom_variable_large}
the responses of the closed-loop systems when employing
the nominal controller, without compensating for the effects
of state measurements quantization. For a fixed, small value
of \(\mu\) (\(\mu = 0.1\)), we note that the states of the closed-loop
system grow unbounded under the nominal control law, as
depicted in Figure~\ref{fixed_zoom_variable_small}. This occurs because the initial condition lies outside the range of the quantizer, which is defined
as \(M\mu\), and thus, the quantizer and control input saturate.
\begin{figure*}[t]
	\centering{
		\subfigure{\includegraphics[width=1\columnwidth]{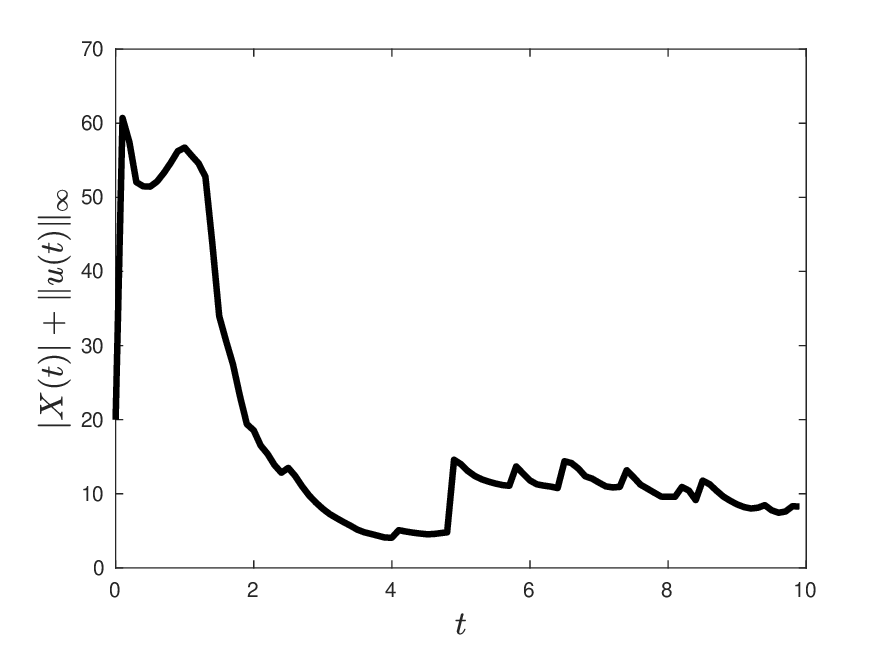} }
		\subfigure{\includegraphics[width=1.05\columnwidth]{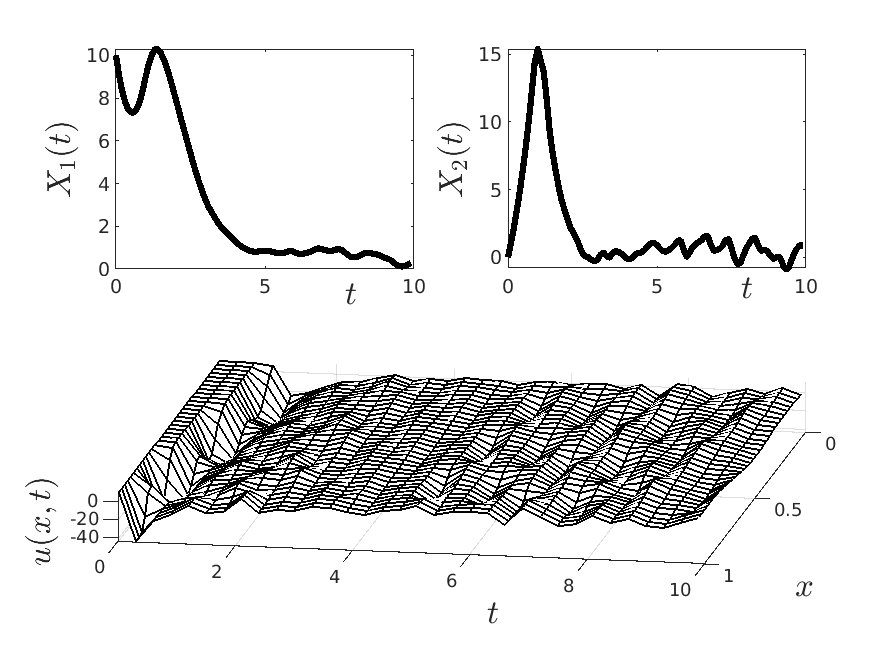} }
		\caption{Left: The norm $|X(t)| + \|u(t)\|_{\infty}$ of the closed-loop system \eqref{pde_representation}--\eqref{pde_representation1}, for $D = 1$, $M = 2$, $\Delta =\dfrac{M}{100}$, under the nominal feedback law
			$U=\mu K\displaystyle\int_{0}^{D}e^{A(D-y)}Bq\left(\dfrac{u(y)}{\mu}\right)dy+\mu K e^{AD}\begin{bmatrix}
				q\left(\dfrac{X_1}{\mu}\right) \quad q\left(\dfrac{X_2}{\mu}\right)
			\end{bmatrix}^T,$
			for fixed $\mu = 100$ and $q$ defined in \eqref{quantizer_matrix}, \eqref{quantizer_simus}, with parameters $\overline{M}=0.6$, $\overline{M}_1 = 2$, $\Omega=0.63$, $T = 2$, and $\mu_0 = 1$. Right: The respective states of the closed-loop system.}
		\label{fixed_zoom_variable_large}
	}
\end{figure*}
Conversely, for a fixed, large value of \(\mu\) (\(\mu = 100\)), the states
of the closed-loop system remain bounded under the nominal control law, owing to the input-to-state stability property of the nominal backstepping controller with respect to
a boundary disturbance, as the quantizer does not saturate.
However, achieving asymptotic stabilization is unattainable,
as illustrated in Figure~\ref{fixed_zoom_variable_large}, due to significant quantization error \(\Delta\mu\). In fact, the system appears to enter a limit cycle
because, due to the quantization effect, the control input becomes negligible when the state lies within a certain region
around zero, leading to state growth (since the open-loop system is unstable), until the quantizer switches to a non-zero
value. This behavior aligns with findings reported for finite-dimensional systems (see, e.g., \cite{liberzon2003hybrid,tarbouriech2011control}), hyperbolic systems
(as discussed in \cite{bekiaris2020hybrid,espitia2017stabilization,tanwani2016input}) and parabolic systems (see \cite{katz2022sampled}),
which do not explicitly aim to compensate for quantization
effects to achieve asymptotic stabilization.
\section{Conclusions and Future Work}\label{conlusion}
Global asymptotic stability of linear systems with input delay and subject to both state and input quantization has been established, thanks to a switched predictor-feedback control law that we introduced. The proof strategy utilized the backstepping method along with small-gain and input-to-state stability arguments. Future works should incorporate state-dependent event-based switching mechanisms since the switching parameter in this paper operates as a time-dependent periodic mechanism. Moreover, consideration should be given to nonlinear time-delay systems, and other classes of PDEs, such as parabolic equations, and cascade parabolic-ODE systems, subject to quantization.	
\bibliographystyle{plain}
\bibliography{TDS_Quantizer_SCL}
\end{document}